\documentclass{article}

\usepackage{amsfonts}
\usepackage{amsmath}
\usepackage{mathtools}
\usepackage{multirow}
\usepackage{epstopdf}
\usepackage{algorithm}
\usepackage{algorithmic}
\usepackage{bm}
\usepackage{microtype}
\usepackage{amssymb}
\usepackage{enumerate}
\usepackage{marvosym}
\usepackage{color}
\usepackage{marvosym}
\usepackage{subcaption}
\usepackage{hyperref}

\newtheorem{theorem}{Theorem}

\newtheorem{lemma}{Lemma}
\newtheorem{proof}{Proof}

\newcommand{\x}{\mathbf{x}}
\newcommand{\y}{\mathbf{y}}

\newcommand{\e}{\mathbf{e}}

\newcommand{\0}{\mathbf{0}}

\newcommand{\E}{\mathbb{E}}
\newcommand{\bO}{{\cal O}}

\newcommand{\F}{\mathcal{F}}

\newcommand{\V}{\mathbf{V}}

\newcommand{\M}{\mathbf{M}}

\newcommand{\G}{\mathbf{G}}

\newcommand{\X}{\mathbf{X}}
\newcommand{\A}{\mathbf{A}}
\newcommand{\Y}{\mathbf{Y}}
\newcommand{\bL}{\mathbf{L}}
\newcommand{\bR}{\mathbf{R}}
\newcommand{\bP}{\mathbf{P}}
\newcommand{\bQ}{\mathbf{Q}}

\newcommand{\<}{\left\langle}
\renewcommand{\>}{\right\rangle}

\usepackage[accepted]{icml2020}

\icmltitlerunning{Convergence Rate Analysis of SOAP with Arbitrary Orthogonal Projection Matrices}

\begin{document}

\onecolumn
\icmltitle{Convergence Rate Analysis of SOAP\\ with Arbitrary Orthogonal Projection Matrices}

\begin{icmlauthorlist}
\icmlauthor{Huan Li}{to}
\icmlauthor{Zhouchen Lin}{goo}
\end{icmlauthorlist}

\icmlaffiliation{to}{Institute of Robotics and Automatic Information Systems, College of Artificial Intelligence, Nankai University, Tianjin, China.\\}
\icmlaffiliation{goo}{National Key Lab of General AI, School of Intelligence Science and Technology, Peking University, Beijing, China.\\}
\icmlcorrespondingauthor{Huan Li and Zhouchen Lin}{lihuanss@nankai.edu.cn, zlin@pku.edu.cn}





\vskip 0.3in



\printAffiliationsAndNotice{}  

\begin{center}
\noindent {\color{red}{This unpublished manuscript is subsumed by our recent paper\\ \textit{Convergence of Rotation-based Matrix Optimizers: A Unified Analysis of SOAP, Conda, and SPlus, arxiv 2609.24522 } }}
\end{center}
\vspace{1em}

\begin{abstract}
In this short note, we establish, for the first time, the convergence rate of SOAP, an efficient and popular matrix-based optimizer for training deep neural networks. Our analysis extends to a more general variant of SOAP that admits arbitrary orthogonal projection matrices and requires only that these matrices be conditionally independent of the current stochastic gradient at each iteration. For example, they may be constructed from information available up to the preceding step.
\end{abstract}

\section{Introduction}

Pre-training large language models incurs substantial computational costs. Over the past decade, the Adam family of algorithms, including AdaGrad \citep{Duchi-2011-jmlr,McMahan-2010-colt}, RMSProp \citep{RMSProp-2012-hinton}, Adam \citep{adam-15-iclr}, and AdamW \citep{adanw-2019-iclr}, has served as the de facto standard for optimizing deep neural networks. These methods treat the network's parameters as a single high-dimensional vector. More recently, matrix-based optimizers, which exploit the inherent matrix structure of deep neural network parameters, have garnered increasing attention and demonstrated the potential to outperform their vector-based counterparts. Prominent examples include Shampoo \citep{shampoo-icml-18}, Muon \citep{muon2024}, SOAP \citep{soap-2025-iclr}, Splus \citep{splus-2025-nips}, and ARO \citep{aro-gong-26}. The latter three approaches can be collectively classified as projection-based optimizers.

Although the convergence behavior of vector-based optimizers has been thoroughly studied in the literature \citep{bottou-2022-tmlr,luo-2020-iclr,lihuan-rmsprop-2024,luo-2022-nips,hong-2024-adam,haochuanli-2023,Li-2025-nips}, the theoretical understanding of matrix-based optimizers remains comparatively underdeveloped. Among matrix-based methods, Muon constitutes a notable exception: owing to its relatively simple structure, its convergence properties have been extensively investigated and are now largely understood \citep{muon-hongmingyi-25,muon-shen-25,muon-chen-25,muon-sato-25}. For more intricate optimizers, recent progress includes the work of \citet{shampoo-li-26}, who established rigorous convergence guarantees for the practical AdamW-style implementation of Shampoo. Nevertheless, with the exception of these two cases, the convergence analysis of other matrix-based optimizers remains an open problem to the best of our knowledge, especially the representative SOAP optimizer. The work by \citet{galore-icml-25} perhaps constitutes the most closely related study, analyzing a variant of the GaLore algorithm \citep{galore-icml-24}. However, their approach replaces the structured SVD-based orthogonal projection employed by GaLore and SOAP with a random projection, and it utilizes momentum stochastic gradient descent (MSGD) within the projected space, whereas GaLore and SOAP leverage the adaptive mechanisms of Adam. Consequently, the analytical techniques developed in \citep{galore-icml-25} are far from sufficient to establish the convergence of the full SOAP algorithm.

In this short note, we address the theoretical behavior of SOAP and establish the following convergence rate
\begin{eqnarray}
\begin{aligned}\label{rate2}
\frac{1}{K}\sum_{k=1}^K\E\left[\left\|\nabla f(\X_k)\right\|_*\right]\leq \bO\left(\sqrt{mn}\sqrt[4]{\frac{\sigma_F^2L\left(f(\X_1)-f^*\right)}{K}}\right),
\end{aligned}
\end{eqnarray} 
where the relevant notations can be found in Section \ref{sec:assumption}. Notably, our analysis extends to a considerably more general variant of SOAP, as presented in Algorithm \ref{gsoap}, which accommodates arbitrary orthogonal projection matrices. In other words,  the projection operators may be generated through any procedure. The only requirement is that, at each iteration, the projection matrices be conditionally independent of the current stochastic gradient. Formally, in the proof for Algorithm \ref{gsoap} we use the identity $\E_k\left[\bP_{k-1}^T\G_k\bQ_{k-1}|\F_{k-1}\right]=\bP_{k-1}^T\E_k\left[\G_k|\F_{k-1}\right]\bQ_{k-1}$. This flexibility is substantial: for instance, the projection matrices may be constructed adaptively from the entire optimization history up to the preceding step—by employing, for example, the eigenvectors of the one-step-delayed preconditioning matrices used in the original SOAP (Algorithm \ref{soap}), or the left and right singular vectors of the momentum or gradient from the previous iteration. Consequently, the theoretical guarantees we establish not only apply directly to the original SOAP algorithm but also shed light on a broader class of projection-based matrix optimizers, including Splus \citep{splus-2025-nips} and ARO \citep{aro-gong-26}, thereby offering a unified analytical framework for understanding the convergence behavior of such methods.

\noindent\begin{minipage}{0.5\textwidth}
\begin{algorithm}[H]
    \caption{SOAP}
    \label{soap}
    \begin{algorithmic}
       \STATE Hyper parameters: $\eta,\theta,\beta,\varepsilon$
       \STATE Initialize $\X_1$, $\M_0=\0$, $\V_0=\0$, $\bP_0$, $\bQ_0$, $\bL_0$, $\bR_0$.
       \FOR{$k=1,2,\cdots,K$}
           \STATE $\G_k=\mbox{GradOracle}(\X_k)$
           \STATE $\M_k=\theta\M_{k-1}+(1-\theta)\G_k$
           \STATE $\G_k'=\bP_{k-1}^T\G_k\bQ_{k-1}$, $\M_k'=\bP_{k-1}^T\M_k\bQ_{k-1}$
           \STATE $\V_k=\beta\V_{k-1}+(1-\beta)(\G_k')^2$
           \STATE $\X_{k+1}=\X_k-\eta\bP_{k-1}\frac{\M_k'}{\sqrt{\V_k+\varepsilon}}\bQ_{k-1}^T$
           \STATE $\bL_k=\beta \bL_{k-1}+(1-\beta)\G_k\G_k^T$
           \STATE $\bR_k=\beta \bR_{k-1}+(1-\beta)\G_k^T\G_k$
           \IF{$k\mod T=0$}
              \STATE $\bP_k=\mbox{Eigenvector}(\bL_k)$
              \STATE $\bQ_k=\mbox{Eigenvector}(\bR_k)$
           \ELSE
              \STATE $\bP_k=\bP_{k-1}$, $\bQ_k=\bQ_{k-1}$
           \ENDIF
       \ENDFOR
    \end{algorithmic}
\end{algorithm}
\end{minipage}
\begin{minipage}{0.5\textwidth}
\begin{algorithm}[H]
    \caption{Generalized SOAP}
    \label{gsoap}
    \begin{algorithmic}
       \STATE Hyper parameters: $\eta,\theta,\beta,\varepsilon$
       \STATE Initialize $\X_1$, $\M_0=\0$, $\V_0=\0$, $\bP_0$, $\bQ_0$.
       \FOR{$k=1,2,\cdots,K$}
           \STATE $\G_k=\mbox{GradOracle}(\X_k)$
           \STATE $\M_k=\theta\M_{k-1}+(1-\theta)\G_k$
           \STATE $\G_k'=\bP_{k-1}^T\G_k\bQ_{k-1}$, $\M_k'=\bP_{k-1}^T\M_k\bQ_{k-1}$
           \STATE $\V_k=\beta\V_{k-1}+(1-\beta)(\G_k')^2$
           \STATE $\X_{k+1}=\X_k-\eta\bP_{k-1}\frac{\M_k'}{\sqrt{\V_k+\varepsilon}}\bQ_{k-1}^T$
           \STATE $\mbox{Generate orthonogal matrices }\bP_k\mbox{ and }\bQ_k$
       \ENDFOR
       \STATE
       \STATE
       \STATE
       \STATE
       \STATE
       \STATE
       \STATE
    \end{algorithmic}
  \end{algorithm}
\end{minipage}

\subsection{Problem Settings, Notations, and Assumptions}\label{sec:assumption}

We study the following nonconvex problem with matrix parameter in this note
\begin{equation}\label{problem}
\min_{\X\in\mathbb R^{m\times n}} f(\X),
\end{equation}
where $f(\X)=\E_{\zeta\in\mathcal{P}}[f(\X;\zeta)]$ and $\zeta$ is the sample drawn from the data distribution $\mathcal{P}$.

We denote vectors by lowercase bold letters and matrices by uppercase bold letters. For vectors, denote $\|\cdot\|$ as the $\ell_2$ Euclidean norm. For matrices, denote $\|\cdot\|_F$, $\|\cdot\|_{op}$, $\|\cdot\|_*$, and $\|\cdot\|_1$ as the Frobenius norm, spectral norm (largest singular value), nuclear norm (sum of singular values), and entrywise $\ell_1$ norm (sum of absolute entries), respectively. Denote $\X_{k,i,j}$, $\X_{k,i,:}$, and $\X_{k,:,j}$ to be the $(i,j)$-th element, $i$-th row, and $j$-th column of matrix $\X$ at iteration $k$, respectively. For matrices $\X$ and $\Y$, denote $\frac{\X}{\Y}$, $\X^2$, and $\sqrt{\X}$ as the element-wise division, square, and square root (different from the matrix power, with a slight abuse of terminology), respectively. Denote $\F_k=\sigma(\zeta_1,\zeta_2,\cdots,\zeta_k)$ to be the sigma field of the stochastic samples up to $k$, denote $\E_{\F_k}[\cdot]$ as the expectation with respect to $\F_k$ and $\E_k[\cdot|\F_{k-1}]$ the conditional expectation with respect to $\zeta_k$ given $\F_{k-1}$. For the sake of brevity, $\E_{\F_K}[\cdot]$ will be denoted as $\E[\cdot]$. 

We make the following assumptions throughout this note:
\begin{enumerate}
\item Smoothness: $\|\nabla f(\Y)-\nabla f(\X)\|_F\leq L\|\Y-\X\|_F, \forall \X,\Y$,
\item Unbiased estimator: $\E_k\left[\G_k\big|\F_{k-1}\right]=\nabla f(\X_k)$,
\item Bounded noise variance: $\E_k\left[\|\G_k-\nabla f(\X_k)\|_F^2\big|\F_{k-1}\right]\leq\sigma_F^2$, \quad$\E_k\left[\|\G_k-\nabla f(\X_k)\|_{op}^2\big|\F_{k-1}\right]\leq\sigma_{op}^2$.
\end{enumerate}

\section{Convergence Rate of SOAP}
We establish the following theorem for the generalized SOAP algorithm presented in Algorithm \ref{gsoap}, which includes the original SOAP as a special case. The rate in (\ref{rate1}) recovers the convergence rate of Adam in the projection space in terms of the entrywise $\ell_1$ norm, and it matches the lower bound for stochastic nonconvex optimization \citep{Arjevani-2023-mp} in the idealized regime where $\|\X\|_1=\Theta(\sqrt{mn})\|\X\|_F$. However, because the $\ell_1$ norm is not unitarily invariant, the orthogonal projection matrices appearing on the left-hand side of (\ref{rate1}) cannot be eliminated. In contrast, the rate in (\ref{rate2}) employs the nuclear norm, which directly measures the gradient magnitude in the primal space. Despite this advantage, the rate in (\ref{rate2}) is $\sqrt{\min\{m,n\}}$ times slower than the rate achieved by AdamW-style Shampoo \citep{shampoo-li-26} and at least $\sqrt{\max\{m,n\}}$ times slower than the theoretical lower bound. In particular, if we choose $\bP_{k-1}$ and $\bQ_{k-1}$ to be the left and right singular vector matrices of $\nabla f(\X_k)$, we obtain $\left\|\bP_{k-1}^T\nabla f(\X_k)\bQ_{k-1}\right\|_1=\left\|\nabla f(\X_k)\right\|_*$ and rate (\ref{rate1}) reduces to rate (\ref{rate2}).
\begin{theorem}\label{main-theorem}
Suppose that Assumptions 1-3 hold. Let $\hat\sigma_F^2=\max\left\{\sigma_F^2,\frac{L\left(f(\X_1)-f^*\right)}{K\gamma^2}\right\}$ with any $\gamma\in(0,1]$, $1-\theta=\sqrt{\frac{L\left(f(\X_1)-f^*\right)}{K\hat\sigma_F^2}}$, $0\leq\beta<1$, $\varepsilon=\sigma_{op}^2$, and $\eta=\sqrt{\frac{\varepsilon\left(f(\X_1)-f^*\right)}{4LK\hat\sigma_F^2}}$. Then for Algorithm \ref{gsoap}, we have
\begin{eqnarray}
\begin{aligned}\label{rate1}
\frac{1}{K}\sum_{k=1}^K\E\left[\left\|\bP_{k-1}^T\nabla f(\X_k)\bQ_{k-1}\right\|_1\right]\leq 10\sqrt{\frac{\hat\sigma_F^2L\left(f(\X_1)-f^*\right)}{K\sigma_{op}^2}}+4\sqrt{mn}\sqrt[4]{\frac{\hat\sigma_F^2L\left(f(\X_1)-f^*\right)}{K}},
\end{aligned}
\end{eqnarray} 
and
\begin{eqnarray}
\begin{aligned}\label{rate2}
\frac{1}{K}\sum_{k=1}^K\E\left[\left\|\nabla f(\X_k)\right\|_*\right]\leq 10\sqrt{\frac{\hat\sigma_F^2L\left(f(\X_1)-f^*\right)}{K\sigma_{op}^2}}+4\sqrt{mn}\sqrt[4]{\frac{\hat\sigma_F^2L\left(f(\X_1)-f^*\right)}{K}}.
\end{aligned}
\end{eqnarray} 
\end{theorem}
The proof presented below follows the technical framework of AdamW established in \citep{Li-2025-nips}. While there is some overlap with the arguments provided therein, we include the complete proof here for the sake of completeness.
\begin{proof}
As the gradient is $L$-Lipschitz, we have
\begin{eqnarray}
\begin{aligned}\label{equ1}
&f(\X_{k+1})-f(\X_k)\\
\leq&\<\nabla f(\X_k),\X_{k+1}-\X_k\>+\frac{L}{2}\|\X_{k+1}-\X_k\|_F^2\\
=&-\eta\<\nabla f(\X_k),\bP_{k-1}\frac{\M_k'}{\sqrt{\V_k+\varepsilon}}\bQ_{k-1}^T\>+\frac{L\eta^2}{2}\left\|\bP_{k-1}\frac{\M_k'}{\sqrt{\V_k+\varepsilon}}\bQ_{k-1}^T\right\|_F^2\\
=&-\eta\<\bP_{k-1}^T\nabla f(\X_k)\bQ_{k-1},\frac{\bP_{k-1}^T\M_k\bQ_{k-1}}{\sqrt{\V_k+\varepsilon}}\>+\frac{L\eta^2}{2}\left\|\bP_{k-1}\frac{\bP_{k-1}^T\M_k\bQ_{k-1}}{\sqrt{\V_k+\varepsilon}}\bQ_{k-1}^T\right\|_F^2\\
\overset{a}=&-\eta\<\frac{\bP_{k-1}^T\nabla f(\X_k)\bQ_{k-1}}{\sqrt[4]{\V_k+\varepsilon}},\frac{\bP_{k-1}^T\M_k\bQ_{k-1}}{\sqrt[4]{\V_k+\varepsilon}}\>+\frac{L\eta^2}{2}\left\|\frac{\bP_{k-1}^T\M_k\bQ_{k-1}}{\sqrt{\V_k+\varepsilon}}\right\|_F^2\\
\leq&-\eta\<\frac{\bP_{k-1}^T\nabla f(\X_k)\bQ_{k-1}}{\sqrt[4]{\V_k+\varepsilon}},\frac{\bP_{k-1}^T\M_k\bQ_{k-1}}{\sqrt[4]{\V_k+\varepsilon}}\>+\frac{L\eta^2}{2\sqrt{\varepsilon}}\left\|\frac{\bP_{k-1}^T\M_k\bQ_{k-1}}{\sqrt[4]{\V_k+\varepsilon}}\right\|_F^2\\
=&-\frac{\eta}{2}\left\|\frac{\bP_{k-1}^T\nabla f(\X_k)\bQ_{k-1}}{\sqrt[4]{\V_k+\varepsilon}}\right\|_F^2-\frac{\eta}{2}\left\|\frac{\bP_{k-1}^T\M_k\bQ_{k-1}}{\sqrt[4]{\V_k+\varepsilon}}\right\|_F^2\\
&+\frac{\eta}{2}\left\|\frac{\bP_{k-1}^T\nabla f(\X_k)\bQ_{k-1}}{\sqrt[4]{\V_k+\varepsilon}}-\frac{\bP_{k-1}^T\M_k\bQ_{k-1}}{\sqrt[4]{\V_k+\varepsilon}}\right\|_F^2+\frac{L\eta^2}{2\sqrt{\varepsilon}}\left\|\frac{\bP_{k-1}^T\M_k\bQ_{k-1}}{\sqrt[4]{\V_k+\varepsilon}}\right\|_F^2
\end{aligned}
\end{eqnarray}
\begin{eqnarray}
\begin{aligned}\notag
\overset{b}\leq&-\frac{\eta}{2}\left\|\frac{\bP_{k-1}^T\nabla f(\X_k)\bQ_{k-1}}{\sqrt[4]{\V_k+\varepsilon}}\right\|_F^2-\frac{\eta}{4}\left\|\frac{\bP_{k-1}^T\M_k\bQ_{k-1}}{\sqrt[4]{\V_k+\varepsilon}}\right\|_F^2+\frac{\eta}{2\sqrt{\varepsilon}}\left\|\bP_{k-1}^T\nabla f(\X_k)\bQ_{k-1}-\bP_{k-1}^T\M_k\bQ_{k-1}\right\|_F^2\\
\overset{c}=&-\frac{\eta}{2}\left\|\frac{\bP_{k-1}^T\nabla f(\X_k)\bQ_{k-1}}{\sqrt[4]{\V_k+\varepsilon}}\right\|_F^2-\frac{\eta}{4}\left\|\frac{\bP_{k-1}^T\M_k\bQ_{k-1}}{\sqrt[4]{\V_k+\varepsilon}}\right\|_F^2+\frac{\eta}{2\sqrt{\varepsilon}}\left\|\nabla f(\X_k)-\M_k\right\|_F^2,
\end{aligned}
\end{eqnarray}
where we use the fact that $\|\cdot\|_F$ is a unitarily invariant norm such that $\|\bP\A\bQ^T\|_F=\|\A\|_F$ for orthogonal matrices $\bP$ and $\bQ$ in $\overset{a}=$ and $\overset{c}=$, and let $\eta\leq\frac{\sqrt{\varepsilon}}{2L}$ in $\overset{b}\leq$. Taking expectation on both sides of (\ref{equ1}) with respect to $\zeta_k$ conditioned on $\F_{k-1}$, multiplying both sides of (\ref{GM-dif-equ}) by $\frac{\eta}{2\sqrt{\varepsilon}(1-\theta)}$, and adding to (\ref{equ1}), we have
\begin{eqnarray}
\begin{aligned}\label{equ3}
&\E_k\left[f(\X_{k+1})-f^*+\frac{\eta\theta}{2\sqrt{\varepsilon}(1-\theta)}\left\|\nabla f(\X_k)-\M_k\right\|_F^2+\frac{\eta}{4}\left\|\frac{\bP_{k-1}^T\M_k\bQ_{k-1}}{\sqrt[4]{\V_k+\varepsilon}}\right\|_F^2\big| \F_{k-1}\right]\\
\leq&f(\X_k)-f^*-\frac{\eta}{2}\E_k\left[\left\|\frac{\bP_{k-1}^T\nabla f(\X_k)\bQ_{k-1}}{\sqrt[4]{\V_k+\varepsilon}}\right\|_F^2\big| \F_{k-1}\right]\\
& +\frac{\eta\theta}{2\sqrt{\varepsilon}(1-\theta)}\left\|\M_{k-1}-\nabla f(\X_{k-1})\right\|_F^2+ \frac{L^2\eta^3}{2\varepsilon(1-\theta)^2}\left\|\frac{\bP_{k-2}^T\M_{k-1}\bQ_{k-2}}{\sqrt[4]{\V_{k-1}+\varepsilon}}\right\|_F^2 + \frac{\eta(1-\theta)\sigma_F^2}{2\sqrt{\varepsilon}}\\
\leq&f(\X_k)-f^*-\frac{\eta}{2}\E_k\left[\left\|\frac{\bP_{k-1}^T\nabla f(\X_k)\bQ_{k-1}}{\sqrt[4]{\V_k+\varepsilon}}\right\|_F^2\big| \F_{k-1}\right]\\
& +\frac{\eta\theta}{2\sqrt{\varepsilon}(1-\theta)}\left\|\M_{k-1}-\nabla f(\X_{k-1})\right\|_F^2+ \frac{\eta}{4}\left\|\frac{\bP_{k-2}^T\M_{k-1}\bQ_{k-2}}{\sqrt[4]{\V_{k-1}+\varepsilon}}\right\|_F^2 + \frac{\eta(1-\theta)\sigma_F^2}{2\sqrt{\varepsilon}},
\end{aligned}
\end{eqnarray}
where we let $\eta^2\leq\frac{\varepsilon(1-\theta)^2}{2L^2} $ in the last inequality. Define 
\begin{eqnarray}
\begin{aligned}\label{wide-v-def}
\widetilde \V_k=\beta\V_{k-1}+(1-\beta)\left(\left(\bP_{k-1}^T\nabla f(\X_k)\bQ_{k-1}\right)^2+\sigma_{op}^2\right).
\end{aligned}
\end{eqnarray}
Since $\bP_{k-1}$ and $\bQ_{k-1}$ are not random variables given $\F_{k-1}$, we have
\begin{eqnarray}
\begin{aligned}\label{equ2}
&\E_k\left[\left|\bP_{k-1,:,i}^T\G_k\bQ_{k-1,:,j}\right|^2\big| \F_{k-1}\right]\\
\overset{d}=&\E_k\left[\left|\bP_{k-1,:,i}^T\nabla f(\X_k)\bQ_{k-1,:,j}\right|^2+\left|\bP_{k-1,:,i}^T\left(\nabla f(\X_k)-\G_k\right)\bQ_{k-1,:,j}\right|^2\Big|\F_{k-1}\right]\\
\overset{e}\leq&\E_k\left[\left|\bP_{k-1,:,i}^T\nabla f(\X_k)\bQ_{k-1,:,j}\right|^2+\left\|\nabla f(\X_k)-\G_k\right\|_{op}^2\Big|\F_{k-1}\right]\\
=&\left|\bP_{k-1,:,i}^T\nabla f(\X_k)\bQ_{k-1,:,j}\right|^2+\E_k\left[\left\|\nabla f(\X_k)-\G_k\right\|_{op}^2\Big|\F_{k-1}\right]\\
\overset{f}\leq&\left|\bP_{k-1,:,i}^T\nabla f(\X_k)\bQ_{k-1,:,j}\right|^2+\sigma_{op}^2, 
\end{aligned}
\end{eqnarray}
and
\begin{eqnarray}
\begin{aligned}\notag
\E_k\left[\V_{k,i,j}\big| \F_{k-1}\right]=&\beta\V_{k-1,i,j}+(1-\beta)\E_k\left[\left|\bP_{k-1,:,i}^T\G_k\bQ_{k-1,:,j}\right|^2\big| \F_{k-1}\right]\\
\leq&\beta\V_{k-1,i,j}+(1-\beta)\left(\left|\bP_{k-1,:,i}^T\nabla f(\X_k)\bQ_{k-1,:,j}\right|^2+\sigma_{op}^2\right)\\
=&\widetilde \V_{k,i,j},
\end{aligned}
\end{eqnarray}
where we use Assumption 2 in $\overset{d}=$, $\<\x,\A\y\>\leq\|\x\|\|\A\y\|\leq\|\x\|\|\A\|_{op}\|\y\|$, $\|\bP_{k-1,:,i}\|=1$, and $\|\bQ_{k-1,:,j}\|=1$ in $\overset{e}\leq$, and Assumption 3 in $\overset{f}\leq$. From the concavity of $-\frac{1}{\sqrt{x}}$, we have
\begin{eqnarray}
\begin{aligned}\notag
\E_k\left[-\frac{\left|\bP_{k-1,:,i}^T\nabla f(\X_k)\bQ_{k-1,:,j}\right|^2}{\sqrt{\V_{k,i,j}+\varepsilon}}\big| \F_{k-1}\right]\leq -\frac{\left|\bP_{k-1,:,i}^T\nabla f(\X_k)\bQ_{k-1,:,j}\right|^2}{\sqrt{\E_k\left[\V_{k,i,j}\big| \F_{k-1}\right]+\varepsilon}}\leq -\frac{\left|\bP_{k-1,:,i}^T\nabla f(\X_k)\bQ_{k-1,:,j}\right|^2}{\sqrt{\widetilde \V_{k,i,j}+\varepsilon}}.
\end{aligned}
\end{eqnarray}
Plugging into (\ref{equ1}) and (\ref{equ3}), we have
\begin{eqnarray}
\begin{aligned}\label{equ4}
&\E_k\left[f(\X_{k+1})\big| \F_{k-1}\right]-f(\X_k)\\
\leq&-\frac{\eta}{2}\left\|\frac{\bP_{k-1}^T\nabla f(\X_k)\bQ_{k-1}}{\sqrt[4]{\widetilde\V_k+\varepsilon}}\right\|_F^2-\frac{\eta}{4}\E_k\left[\left\|\frac{\bP_{k-1}^T\M_k\bQ_{k-1}}{\sqrt[4]{\V_k+\varepsilon}}\right\|_F^2\big| \F_{k-1}\right]+\frac{\eta}{2\sqrt{\varepsilon}}\E_k\left[\left\|\nabla f(\X_k)-\M_k\right\|_F^2\big| \F_{k-1}\right]
\end{aligned}
\end{eqnarray}
and
\begin{eqnarray}
\begin{aligned}\label{equ5}
&\E_k\left[f(\X_{k+1})-f^*+\frac{\eta\theta}{2\sqrt{\varepsilon}(1-\theta)}\left\|\nabla f(\X_k)-\M_k\right\|_F^2+\frac{\eta}{4}\left\|\frac{\bP_{k-1}^T\M_k\bQ_{k-1}}{\sqrt[4]{\V_k+\varepsilon}}\right\|_F^2\big| \F_{k-1}\right]\\
\leq&f(\X_k)-f^*-\frac{\eta}{2}\left\|\frac{\bP_{k-1}^T\nabla f(\X_k)\bQ_{k-1}}{\sqrt[4]{\widetilde\V_k+\varepsilon}}\right\|_F^2\\
& +\frac{\eta\theta}{2\sqrt{\varepsilon}(1-\theta)}\left\|\M_{k-1}-\nabla f(\X_{k-1})\right\|_F^2+ \frac{\eta}{4}\left\|\frac{\bP_{k-2}^T\M_{k-1}\bQ_{k-2}}{\sqrt[4]{\V_{k-1}+\varepsilon}}\right\|_F^2 + \frac{\eta(1-\theta)\sigma_F^2}{2\sqrt{\varepsilon}}.
\end{aligned}
\end{eqnarray}
Taking expectation with respect to $\F_{k-1}$ and summing (\ref{equ4}) with $k=1$ and (\ref{equ5}) over $k=2,3,\cdots,K$, we have
\begin{eqnarray}
\begin{aligned}\label{equ6}
&\E_{\F_K}\left[f(\X_{K+1})-f^*+\frac{\eta\theta}{2\sqrt{\varepsilon}(1-\theta)}\left\|\nabla f(\X_K)-\M_K\right\|_F^2+\frac{\eta}{4}\left\|\frac{\bP_{K-1}^T\M_K\bQ_{K-1}}{\sqrt[4]{\V_K+\varepsilon}}\right\|_F^2\right]\\
\leq&f(\X_1)-f^*-\frac{\eta}{2}\sum_{k=1}^K\E_{\F_{k-1}}\left[\left\|\frac{\bP_{k-1}^T\nabla f(\X_k)\bQ_{k-1}}{\sqrt[4]{\widetilde\V_k+\varepsilon}}\right\|_F^2\right]\\
& +\left(\frac{\eta}{2\sqrt{\varepsilon}}+\frac{\eta\theta}{2\sqrt{\varepsilon}(1-\theta)}\right)\E_{\F_1}\left[\left\|\M_1-\nabla f(\X_1)\right\|_F^2\right]+ \frac{(K-1)\eta(1-\theta)\sigma_F^2}{2\sqrt{\varepsilon}}\\
=&f(\X_1)\hspace*{-0.05cm}-\hspace*{-0.05cm}f^*\hspace*{-0.05cm}-\hspace*{-0.05cm}\frac{\eta}{2}\hspace*{-0.05cm}\sum_{k=1}^K\hspace*{-0.05cm}\E_{\F_{k-1}}\hspace*{-0.15cm}\left[\left\|\frac{\bP_{k-1}^T\nabla f(\X_k)\bQ_{k-1}}{\sqrt[4]{\widetilde\V_k+\varepsilon}}\right\|_F^2\right] \hspace*{-0.1cm}+\hspace*{-0.05cm}\frac{\eta}{2\sqrt{\varepsilon}(1\hspace*{-0.05cm}-\hspace*{-0.05cm}\theta)}\E_{\F_1}\hspace*{-0.15cm}\left[\left\|\M_1\hspace*{-0.1cm}-\hspace*{-0.1cm}\nabla f(\X_1)\right\|_F^2\right]\hspace*{-0.05cm}+\hspace*{-0.05cm} \frac{(\hspace*{-0.05cm}K\hspace*{-0.05cm}-\hspace*{-0.05cm}1)\eta(1\hspace*{-0.05cm}-\hspace*{-0.05cm}\theta)\sigma_F^2}{2\sqrt{\varepsilon}}.\hspace*{-0.65cm}
\end{aligned}
\end{eqnarray}
As the gradient is $L$-Lipschitz, we have
\begin{eqnarray}
\begin{aligned}\notag
&f^*\leq f\left(\X-\frac{1}{L}\nabla f(\X)\right)\leq f(\X)-\frac{1}{L}\<\nabla f(\X),\nabla f(\X)\>+\frac{L}{2}\left\|\frac{1}{L}\nabla f(\X)\right\|_F^2=f(\X)-\frac{1}{2L}\left\|\nabla f(\X)\right\|_F^2.
\end{aligned}
\end{eqnarray}
Using the recursion of $\M_1$ and $\M_0=\0$, we have
\begin{eqnarray}
\begin{aligned}\notag
\E_{\F_1}\left[\left\|\nabla f(\X_1)-\M_1\right\|_F^2\right]=&\E_{\F_1}\left[\left\|\theta\nabla f(\X_1)+(1-\theta)\left(\nabla f(\X_1)-\G_1\right)\right\|_F^2\right]\\
=&\theta^2\left\|\nabla f(\X_1)\right\|_F^2+(1-\theta)^2\E_{\F_1}\left[\left\|\nabla f(\X_1)-\G_1\right\|_F^2\right]\\
\leq& 2L\left(f(\X_1)-f^*\right) + (1-\theta)^2\sigma_F^2.
\end{aligned}
\end{eqnarray}
Plugging into (\ref{equ6}), we have
\begin{eqnarray}
\begin{aligned}\notag
&\E_{\F_K}\left[f(\X_{K+1})-f^*+\frac{\eta\theta}{2\sqrt{\varepsilon}(1-\theta)}\left\|\nabla f(\X_K)-\M_K\right\|_F^2+\frac{\eta}{4}\left\|\frac{\bP_{K-1}^T\M_K\bQ_{K-1}}{\sqrt[4]{\V_K+\varepsilon}}\right\|_F^2\right]\\
\leq&f(\X_1)-f^*-\frac{\eta}{2}\sum_{k=1}^K\E_{\F_{k-1}}\left[\left\|\frac{\bP_{k-1}^T\nabla f(\X_k)\bQ_{k-1}}{\sqrt[4]{\widetilde\V_k+\varepsilon}}\right\|_F^2\right] +\frac{L\eta}{\sqrt{\varepsilon}(1-\theta)}\left(f(\X_1)-f^*\right)+ \frac{K\eta(1-\theta)\sigma_F^2}{2\sqrt{\varepsilon}}\\
\leq&f(\X_1)-f^*-\frac{\eta}{2}\sum_{k=1}^K\E_{\F_{k-1}}\left[\left\|\frac{\bP_{k-1}^T\nabla f(\X_k)\bQ_{k-1}}{\sqrt[4]{\widetilde\V_k+\varepsilon}}\right\|_F^2\right] +\frac{L\eta}{\sqrt{\varepsilon}(1-\theta)}\left(f(\X_1)-f^*\right)+ \frac{K\eta(1-\theta)\hat\sigma_F^2}{2\sqrt{\varepsilon}},
\end{aligned}
\end{eqnarray}
where we denote $\hat\sigma_F^2=\max\left\{\sigma_F^2,\frac{L\left(f(\X_1)-f^*\right)}{K\gamma^2}\right\}$ with any $\gamma\in(0,1]$. So we have
\begin{eqnarray}
\begin{aligned}\notag
\sum_{k=1}^K\E_{\F_{k-1}}\left[\left\|\frac{\bP_{k-1}^T\nabla f(\X_k)\bQ_{k-1}}{\sqrt[4]{\widetilde\V_k+\varepsilon}}\right\|_F^2\right]\leq&\frac{2}{\eta}\left(f(\X_1)-f^*\right) +\frac{2L}{\sqrt{\varepsilon}(1-\theta)}\left(f(\X_1)-f^*\right)+ \frac{K(1-\theta)\hat\sigma_F^2}{\sqrt{\varepsilon}}\\
=&7\sqrt{\frac{K\hat\sigma_F^2L\left(f(\X_1)-f^*\right)}{\varepsilon}}
\end{aligned}
\end{eqnarray}
by letting $1-\theta=\sqrt{\frac{L\left(f(\X_1)-f^*\right)}{K\hat\sigma_F^2}}$ and $\eta=\sqrt{\frac{\varepsilon\left(f(\X_1)-f^*\right)}{4LK\hat\sigma_F^2}}$, which also satisfies $\eta\leq\frac{\sqrt{\varepsilon}}{2L}$ and $\eta^2\leq\frac{\varepsilon(1-\theta)^2}{4L^2}\leq\frac{\varepsilon(1-\theta)^2}{2L^2}$. Using Holder's inequality and Lemma \ref{second-moment-lemma}, we have
\begin{eqnarray}
\begin{aligned}\notag
&\left(\sum_{k=1}^K\E_{\F_{k-1}}\left[\left\|\bP_{k-1}^T\nabla f(\X_k)\bQ_{k-1}\right\|_1\right]\right)^2=\left(\sum_{k=1}^K\sum_{i=1}^m\sum_{j=1}^n\E_{\F_{k-1}}\left[\left|\bP_{k-1,:,i}^T\nabla f(\X_k)\bQ_{k-1,:,j}\right|\right]\right)^2\\
\leq&\left(\sum_{k=1}^K\sum_{i=1}^m\sum_{j=1}^n\E_{\F_{k-1}}\left[\frac{\left|\bP_{k-1,:,i}^T\nabla f(\X_k)\bQ_{k-1,:,j}\right|^2}{\sqrt{\widetilde\V_{k,i,j}+\varepsilon}}\right]\right)\left(\sum_{k=1}^K\sum_{i=1}^m\sum_{j=1}^n\E_{\F_{k-1}}\left[\sqrt{\widetilde\V_{k,i,j}+\varepsilon}\right]\right)\\
\leq&\left(\sum_{k=1}^K\E_{\F_{k-1}}\left[\left\|\frac{\bP_{k-1}^T\nabla f(\X_t)\bQ_{k-1}}{\sqrt[4]{\widetilde\V_k+\varepsilon}}\right\|_F^2\right]\right)\left(Kmn\sqrt{\sigma_{op}^2+\varepsilon}+2\sum_{k=1}^K\E_{\F_{k-1}}\left[\left\|\frac{\bP_{k-1}^T\nabla f(\X_k)\bQ_{k-1}}{\sqrt[4]{\widetilde\V_k+\varepsilon}}\right\|_F^2\right]\right)\\
\leq&\left(7\sqrt{\frac{K\hat\sigma_F^2L\left(f(\X_1)-f^*\right)}{\varepsilon}}\right)\left(Kmn\sqrt{\sigma_{op}^2+\varepsilon}+14\sqrt{\frac{K\hat\sigma_F^2L\left(f(\X_1)-f^*\right)}{\varepsilon}}\right)
\end{aligned}
\end{eqnarray}
and
\begin{eqnarray}
\begin{aligned}\notag
\frac{1}{K}\sum_{k=1}^K\E_{\F_{k-1}}\left[\left\|\bP_{k-1}^T\nabla f(\X_k)\bQ_{k-1}\right\|_1\right]
\leq&10\sqrt{\frac{\hat\sigma_F^2L\left(f(\X_1)-f^*\right)}{K\varepsilon}}+\sqrt{7mn}\sqrt[4]{\left(\sigma_{op}^2+\varepsilon\right)\frac{\hat\sigma_F^2L\left(f(\X_1)-f^*\right)}{K\varepsilon}}\\
\leq&10\sqrt{\frac{\hat\sigma_F^2L\left(f(\X_1)-f^*\right)}{K\sigma_{op}^2}}+4\sqrt{mn}\sqrt[4]{\frac{\hat\sigma_F^2L\left(f(\X_1)-f^*\right)}{K}}
\end{aligned}
\end{eqnarray}
by letting $\varepsilon=\sigma_{op}^2$. From Lemma \ref{nuclear-l1-norms}, we have
\begin{eqnarray}
\begin{aligned}\notag
\left\|\bP_{k-1}^T\nabla f(\X_k)\bQ_{k-1}\right\|_1\geq \left\|\bP_{k-1}^T\nabla f(\X_k)\bQ_{k-1}\right\|_*=\left\|\nabla f(\X_k)\right\|_*,
\end{aligned}
\end{eqnarray}
where we use the fact that nuclear norm is a unitarily invariant norm in the last equation. 
\end{proof}

The following lemma is closely similar to \citep[Lemma 4]{Li-2025-nips} and we list the proof here only for the sake of completeness. 
\begin{lemma}\label{momentum-bound-lemma}
Suppose that Assumptions 1-3 hold. Then for Algorithm \ref{gsoap}, we have
\begin{eqnarray}
\begin{aligned}\label{GM-dif-equ}
&\E_k\left[\left\|\nabla f(\X_k)-\M_k\right\|_F^2\big|\F_{k-1}\right]\\
\leq&\theta\left\|\M_{k-1}-\nabla f(\X_{k-1})\right\|_F^2 + \frac{L^2\eta^2}{(1-\theta)\sqrt{\varepsilon}}\left\|\frac{\bP_{k-2}^T\M_{k-1}\bQ_{k-2}}{\sqrt[4]{\V_{k-1}+\varepsilon}}\right\|_F^2 + (1-\theta)^2\sigma_F^2.
\end{aligned}
\end{eqnarray}
\end{lemma}
\begin{proof}
Denoting $\Gamma_k=\G_k-\nabla f(\X_k)$, we have $\E_k\left[\Gamma_k\big|\F_{k-1}\right]=0$ and $\E_k\left[\left\|\Gamma_k\right\|_F^2\big|\F_{k-1}\right]\leq\sigma_F^2$. From the update of $\M_k$, we have
\begin{eqnarray}
\begin{aligned}\notag
\M_k-\nabla f(\X_k)=& \theta \M_{k-1} + (1-\theta)\G_k - \nabla f(\X_k)\\
=& \theta \left(\M_{k-1}-\nabla f(\X_{k-1})\right) + (1-\theta)\left(\nabla f(\X_k)+\Gamma_k\right) - \nabla f(\X_k) + \theta\nabla f(\X_{k-1})\\
=& \theta \left(\M_{k-1}-\nabla f(\X_{k-1})\right) + (1-\theta)\Gamma_k - \theta\left(\nabla f(\X_k)-\nabla f(\X_{k-1})\right)
\end{aligned}
\end{eqnarray}
and
\begin{eqnarray}
\begin{aligned}\notag
&\E_k\left[\left\|\nabla f(\X_k)-\M_k\right\|_F^2\big|\F_{k-1}\right]\\
=&\left\|\theta \left(\M_{k-1}-\nabla f(\X_{k-1})\right) - \theta\left(\nabla f(\X_k)-\nabla f(\X_{k-1})\right)\right\|_F^2 + (1-\theta)^2\E_k\left[\left\|\Gamma_k\right\|_F^2\big|\F_{k-1}\right]\\
\leq& \theta^2\hspace*{-0.05cm}\left(\hspace*{-0.05cm}1\hspace*{-0.05cm}+\hspace*{-0.05cm}\frac{1\hspace*{-0.05cm}-\hspace*{-0.05cm}\theta}{\theta}\right)\hspace*{-0.05cm}\left\|\M_{k-1}\hspace*{-0.05cm}-\hspace*{-0.05cm}\nabla f(\X_{k-1})\right\|_F^2 \hspace*{-0.05cm}+\hspace*{-0.05cm} \theta^2\hspace*{-0.05cm}\left(\hspace*{-0.05cm}1\hspace*{-0.05cm}+\hspace*{-0.05cm}\frac{\theta}{1\hspace*{-0.05cm}-\hspace*{-0.05cm}\theta}\right)\hspace*{-0.05cm}\left\|\nabla f(\X_k)\hspace*{-0.05cm}-\hspace*{-0.05cm}\nabla f(\X_{k-1})\right\|_F^2 \hspace*{-0.05cm}+\hspace*{-0.05cm} (1\hspace*{-0.05cm}-\hspace*{-0.05cm}\theta)^2\E_k\hspace*{-0.05cm}\left[\left\|\Gamma_k\right\|_F^2\big|\F_{k-1}\right]\\
\leq& \theta\left\|\M_{k-1}-\nabla f(\X_{k-1})\right\|_F^2 + \frac{1}{1-\theta}\left\|\nabla f(\X_k)-\nabla f(\X_{k-1})\right\|_F^2 + (1-\theta)^2\sigma_F^2\\
\leq& \theta\left\|\M_{k-1}-\nabla f(\X_{k-1})\right\|_F^2 + \frac{L^2}{1-\theta}\left\|\X_k-\X_{k-1}\right\|_F^2 + (1-\theta)^2\sigma_F^2\\
=& \theta\left\|\M_{k-1}-\nabla f(\X_{k-1})\right\|_F^2 + \frac{L^2\eta^2}{1-\theta}\left\|\bP_{k-2}\frac{\bP_{k-2}^T\M_{k-1}\bQ_{k-2}}{\sqrt{\V_{k-1}+\varepsilon}}\bQ_{k-2}^T\right\|_F^2 + (1-\theta)^2\sigma_F^2\\
=& \theta\left\|\M_{k-1}-\nabla f(\X_{k-1})\right\|_F^2 + \frac{L^2\eta^2}{1-\theta}\left\|\frac{\bP_{k-2}^T\M_{k-1}\bQ_{k-2}}{\sqrt{\V_{k-1}+\varepsilon}}\right\|_F^2 + (1-\theta)^2\sigma_F^2\\
\leq& \theta\left\|\M_{k-1}-\nabla f(\X_{k-1})\right\|_F^2 + \frac{L^2\eta^2}{(1-\theta)\sqrt{\varepsilon}}\left\|\frac{\bP_{k-2}^T\M_{k-1}\bQ_{k-2}}{\sqrt[4]{\V_{k-1}+\varepsilon}}\right\|_F^2 + (1-\theta)^2\sigma_F^2.
\end{aligned}
\end{eqnarray}
\end{proof} 

The following lemma is adapted from \citep[Lemma 8]{lihuan-rmsprop-2024} and \citep[Lemma 5]{Li-2025-nips}, with modifications introduced to accommodate the projection-based setting of the SOAP algorithm.
\begin{lemma}\label{second-moment-lemma}
Suppose that Assumptions 2-3 hold. Let $0\leq\beta< 1$ and $\V_0=\0$. Then we have
\begin{eqnarray}
\begin{aligned}\notag
\sum_{k=1}^K\sum_{i=1}^m\sum_{j=1}^n \E_{\F_{k-1}}\left[\sqrt{\widetilde\V_{k,i,j}+\varepsilon}\right]\leq Kmn\sqrt{\sigma_{op}^2+\varepsilon}+2\sum_{t=1}^K\E_{\F_{t-1}}\left[\left\|\frac{\bP_{t-1}^T\nabla f(\X_t)\bQ_{t-1}}{\sqrt[4]{\widetilde\V_t+\varepsilon}}\right\|_F^2\right].
\end{aligned}
\end{eqnarray}
\end{lemma}
\begin{proof}
From the definition of $\widetilde\V_k$ in (\ref{wide-v-def}), we have
\begin{eqnarray}
\hspace*{-2cm}\begin{aligned}\notag
&\E_{\F_{k-1}}\left[\sqrt{\widetilde\V_{k,i,j}+\varepsilon}\right]\\
=&\E_{\F_{k-1}}\left[\sqrt{\beta\V_{k-1,i,j}+(1-\beta)\left(\left|\bP_{k-1,:,i}^T\nabla f(\X_k)\bQ_{k-1,:,j}\right|^2+\sigma_{op}^2\right)+\varepsilon}\right]
\end{aligned}
\end{eqnarray}
\begin{eqnarray}
\begin{aligned}\notag
=&\E_{\F_{k-1}}\left[\frac{\beta\V_{k-1,i,j}+(1-\beta)\sigma_{op}^2+\varepsilon}{\sqrt{\beta\V_{k-1,i,j}+(1-\beta)\left(\left|\bP_{k-1,:,i}^T\nabla f(\X_k)\bQ_{k-1,:,j}\right|^2+\sigma_{op}^2\right)+\varepsilon}}\right]\\
& + \E_{\F_{k-1}}\left[\frac{(1-\beta)\left|\bP_{k-1,:,i}^T\nabla f(\X_k)\bQ_{k-1,:,j}\right|^2}{\sqrt{\beta\V_{k-1,i,j}+(1-\beta)\left(\left|\bP_{k-1,:,i}^T\nabla f(\X_k)\bQ_{k-1,:,j}\right|^2+\sigma_{op}^2\right)+\varepsilon}}\right]\\
\leq&\E_{\F_{k-1}}\left[\sqrt{\beta\V_{k-1,i,j}+(1-\beta)\sigma_{op}^2+\varepsilon}\right]+(1-\beta)\E_{\F_{k-1}}\left[\frac{\left|\bP_{k-1,:,i}^T\nabla f(\X_k)\bQ_{k-1,:,j}\right|^2}{\sqrt{\widetilde\V_{k,i,j}+\varepsilon}}\right].
\end{aligned}
\end{eqnarray}
Consider the first part in the general case. From the recursion of $\V_{k-t,i,j}$, we have
\begin{eqnarray}
\begin{aligned}\notag
&\E_{\F_{k-t}}\left[\sqrt{\beta^t\V_{k-t,i,j}+(1-\beta^t)\sigma_{op}^2+\varepsilon}\right]\\
=&\E_{\F_{k-t}}\left[\sqrt{\beta^{t+1}\V_{k-t-1,i,j}+\beta^t(1-\beta)\left|\bP_{k-t-1,:,i}^T\G_{k-t}\bQ_{k-t-1,:,j}\right|^2+(1-\beta^t)\sigma_{op}^2+\varepsilon}\right]\\
=&\E_{\F_{k-t-1}}\left[\E_{k-t}\left[\sqrt{\beta^{t+1}\V_{k-t-1,i,j}+\beta^t(1-\beta)\left|\bP_{k-t-1,:,i}^T\G_{k-t}\bQ_{k-t-1,:,j}\right|^2+(1-\beta^t)\sigma_{op}^2+\varepsilon}\Big|\F_{k-t-1}\right]\right]\\
\overset{a}\leq&\E_{\F_{k-t-1}}\left[\sqrt{\beta^{t+1}\V_{k-t-1,i,j}+\beta^t(1-\beta)\E_{k-t}\left[\left|\bP_{k-t-1,:,i}^T\G_{k-t}\bQ_{k-t-1,:,j}\right|^2\Big|\F_{k-t-1}\right]+(1-\beta^t)\sigma_{op}^2+\varepsilon}\right]\\
\overset{b}\leq&\E_{\F_{k-t-1}}\left[\sqrt{\beta^{t+1}\V_{k-t-1,i,j}+\beta^t(1-\beta)\left(\left|\bP_{k-t-1,:,i}^T\nabla f(\X_{k-t})\bQ_{k-t-1,:,j}\right|^2+\sigma_{op}^2\right)+(1-\beta^t)\sigma_{op}^2+\varepsilon}\right]\\
=&\E_{\F_{k-t-1}}\left[\sqrt{\beta^{t+1}\V_{k-t-1,i,j}+\beta^t(1-\beta)\left|\bP_{k-t-1,:,i}^T\nabla f(\X_{k-t})\bQ_{k-t-1,:,j}\right|^2+(1-\beta^{t+1})\sigma_{op}^2+\varepsilon}\right]\\
=&\E_{\F_{k-t-1}}\left[\frac{\beta^{t+1}\V_{k-t-1,i,j}+(1-\beta^{t+1})\sigma_{op}^2+\varepsilon}{\sqrt{\beta^{t+1}\V_{k-t-1,i,j}+\beta^t(1-\beta)\left|\bP_{k-t-1,:,i}^T\nabla f(\X_{k-t})\bQ_{k-t-1,:,j}\right|^2+(1-\beta^{t+1})\sigma_{op}^2+\varepsilon}}\right]\\
& + \E_{\F_{k-t-1}}\left[\frac{\beta^t(1-\beta)\left|\bP_{k-t-1,:,i}^T\nabla f(\X_{k-t})\bQ_{k-t-1,:,j}\right|^2}{\sqrt{\beta^{t+1}\V_{k-t-1,i,j}+\beta^t(1-\beta)\left|\bP_{k-t-1,:,i}^T\nabla f(\X_{k-t})\bQ_{k-t-1,:,j}\right|^2+(1-\beta^{t+1})\sigma_{op}^2+\varepsilon}}\right]\\
\leq&\E_{\F_{k-t-1}}\left[\sqrt{\beta^{t+1}\V_{k-t-1,i,j}+(1-\beta^{t+1})\sigma_{op}^2+\varepsilon}\right]\\
& + \E_{\F_{k-t-1}}\left[\frac{\beta^t(1-\beta)\left|\bP_{k-t-1,:,i}^T\nabla f(\X_{k-t})\bQ_{k-t-1,:,j}\right|^2}{\sqrt{\beta^{t+1}\V_{k-t-1,i,j}+\beta^t(1-\beta)\left|\bP_{k-t-1,:,i}^T\nabla f(\X_{k-t})\bQ_{k-t-1,:,j}\right|^2+(\beta^t-\beta^{t+1})\sigma_{op}^2+\beta^t\varepsilon}}\right]\\
=&\E_{\F_{k-t-1}}\hspace*{-0.1cm}\left[\sqrt{\beta^{t+1}\V_{k-t-1,i,j}+(1-\beta^{t+1})\sigma_{op}^2+\varepsilon}\right]+\sqrt{\beta^t}(1-\beta)\E_{\F_{k-t-1}}\hspace*{-0.1cm}\left[\frac{\left|\bP_{k-t-1,:,i}^T\nabla f(\X_{k-t})\bQ_{k-t-1,:,j}\right|^2}{\sqrt{\widetilde\V_{k-t,i,j}+\varepsilon}}\right]\hspace*{-0.1cm},
\end{aligned}
\end{eqnarray}
where we use the concavity of $\sqrt{x}$ in $\overset{a}\leq$ and a similar induction to (\ref{equ2}) in $\overset{b}\leq$. Applying the above inequality recursively for $t=1,2,\cdots,k-1$, we have
\begin{eqnarray}
\begin{aligned}\notag
&\E_{\F_{k-1}}\left[\sqrt{\beta\V_{k-1,i,j}+(1-\beta)\sigma_{op}^2+\varepsilon}\right]\\
\leq&\sqrt{\beta^k\V_{0,i,j}+(1-\beta^k)\sigma_{op}^2+\varepsilon}+\sum_{t=1}^{k-1}\sqrt{\beta^t}(1-\beta)\E_{\F_{k-t-1}}\left[\frac{\left|\bP_{k-t-1,:,i}^T\nabla f(\X_{k-t})\bQ_{k-t-1,:,j}\right|^2}{\sqrt{\widetilde\V_{k-t,i,j}+\varepsilon}}\right]\\
=&\sqrt{\beta^k\V_{0,i,j}+(1-\beta^k)\sigma_{op}^2+\varepsilon}+\sum_{t=1}^{k-1}\sqrt{\beta^{k-t}}(1-\beta)\E_{\F_{t-1}}\left[\frac{\left|\bP_{t-1,:,i}^T\nabla f(\X_t)\bQ_{t-1,:,j}\right|^2}{\sqrt{\widetilde\V_{t,i,j}+\varepsilon}}\right]
\end{aligned}
\end{eqnarray}
and
\begin{eqnarray}
\begin{aligned}\notag
\E_{\F_{k-1}}\left[\sqrt{\widetilde\V_{k,i,j}+\varepsilon}\right]\leq& \sqrt{\beta^k\V_{0,i,j}+(1-\beta^k)\sigma_{op}^2+\varepsilon}+\sum_{t=1}^k\sqrt{\beta^{k-t}}(1-\beta)\E_{\F_{t-1}}\left[\frac{\left|\bP_{t-1,:,i}^T\nabla f(\X_t)\bQ_{t-1,:,j}\right|^2}{\sqrt{\widetilde\V_{t,i,j}+\varepsilon}}\right]\\
\leq&\sqrt{\sigma_{op}^2+\varepsilon}+\sum_{t=1}^k\sqrt{\beta^{k-t}}(1-\beta)\E_{\F_{t-1}}\left[\frac{\left|\bP_{t-1,:,i}^T\nabla f(\X_t)\bQ_{t-1,:,j}\right|^2}{\sqrt{\widetilde\V_{t,i,j}+\varepsilon}}\right],
\end{aligned}
\end{eqnarray}
where we use $\V_0=\0$. Summing over $i=1,2,\cdots,m$, $j=1,2,\cdots,n$, and $k=1,2,\cdots,K$, we have
\begin{eqnarray}
\begin{aligned}\notag
&\sum_{k=1}^K\sum_{i=1}^m\sum_{j=1}^n \E_{\F_{k-1}}\left[\sqrt{\widetilde\V_{k,i,j}+\varepsilon}\right]\\
\leq& Kmn\sqrt{\sigma_{op}^2+\varepsilon}+\sum_{k=1}^K\sum_{t=1}^k\sqrt{\beta^{k-t}}(1-\beta)\sum_{i=1}^m\sum_{j=1}^n\E_{\F_{t-1}}\left[\frac{\left|\bP_{t-1,:,i}^T\nabla f(\X_t)\bQ_{t-1,:,j}\right|^2}{\sqrt{\widetilde\V_{t,i,j}+\varepsilon}}\right]\\
=& Kmn\sqrt{\sigma_{op}^2+\varepsilon}+\sum_{t=1}^K\sum_{k=t}^K\sqrt{\beta^{k-t}}(1-\beta)\sum_{i=1}^m\sum_{j=1}^n\E_{\F_{t-1}}\left[\frac{\left|\bP_{t-1,:,i}^T\nabla f(\X_t)\bQ_{t-1,:,j}\right|^2}{\sqrt{\widetilde\V_{t,i,j}+\varepsilon}}\right]\\
\leq& Kmn\sqrt{\sigma_{op}^2+\varepsilon}+\frac{1-\beta}{1-\sqrt{\beta}}\sum_{t=1}^K\sum_{i=1}^m\sum_{j=1}^n\E_{\F_{t-1}}\left[\frac{\left|\bP_{t-1,:,i}^T\nabla f(\X_t)\bQ_{t-1,:,j}\right|^2}{\sqrt{\widetilde\V_{t,i,j}+\varepsilon}}\right]\\
=& Kmn\sqrt{\sigma_{op}^2+\varepsilon}+(1+\sqrt{\beta})\sum_{t=1}^K\sum_{i=1}^m\sum_{j=1}^n\E_{\F_{t-1}}\left[\frac{\left|\bP_{t-1,:,i}^T\nabla f(\X_t)\bQ_{t-1,:,j}\right|^2}{\sqrt{\widetilde\V_{t,i,j}+\varepsilon}}\right]\\
=& Kmn\sqrt{\sigma_{op}^2+\varepsilon}+(1+\sqrt{\beta})\sum_{t=1}^K\E_{\F_{t-1}}\left[\left\|\frac{\bP_{t-1}^T\nabla f(\X_t)\bQ_{t-1}}{\sqrt[4]{\widetilde\V_t+\varepsilon}}\right\|_F^2\right].
\end{aligned}
\end{eqnarray}
\end{proof}

The next lemma is used to transform rate (\ref{rate1}) to (\ref{rate2}).
\begin{lemma}\label{nuclear-l1-norms}
For any matrix $\A$, we have $\|\A\|_*\leq \|\A\|_1$.
\end{lemma}
\begin{proof}
Let $\e_i$ denote the $i$-th standard basis vector.
\begin{eqnarray}
\begin{aligned}\notag
\|\A\|_*=\left\|\sum_{i,j}\A_{ij}\e_i\e_j^T\right\|_*\leq \sum_{i,j}|\A_{ij}|\left\|\e_i\e_j^T\right\|_*=\sum_{i,j}|\A_{ij}|=\|\A\|_1.
\end{aligned}
\end{eqnarray}
\end{proof}

\bibliography{SOAP}
\bibliographystyle{icml2020}

\end{document}